\documentclass[a4paper,11pt,reqno]{amsart}

\usepackage{amsmath,amssymb,amsthm,mathtools}
\usepackage{mathrsfs}
\usepackage{bm}
\usepackage{geometry}
\usepackage{enumitem}
\usepackage{hyperref}
\usepackage{color}
\usepackage[pagewise,displaymath,mathlines]{lineno}

\newtheorem{theorem}{Theorem}[section]

\newtheorem{proposition}[theorem]{Proposition}
\newtheorem{lemma}[theorem]{Lemma}
\newtheorem{corollary}[theorem]{Corollary}
\newtheorem{remark}[theorem]{Remark}
\newtheorem{definition}[theorem]{Definition}

\newtheorem{notation}[theorem]{Notations}
\newcommand{\C}{\mathbb C}

\newcommand{\ord}{\operatorname{ord}}
\newcommand{\detm}{\operatorname{det}}

\begin{document}

%\linenumbers
	
	\bigskip

	\title[Complex evolutoids of holomorphic plane curves]{Complex evolutoids of holomorphic plane curves}

 \author[Deolindo-Silva]{Jorge Luiz Deolindo-Silva}
 \thanks{Corresponding author: Jorge Luiz Deolindo-Silva
(\texttt{jorge.deolindo@ufu.br}).}
 
	 \address[Jorge]{Universidade Federal de Uberlândia - Brazil.}
	 \email{jorge.deolindo@ufu.br}

	 \author[Salarinoghabi]{Mostafa Salarinoghabi}
	 \address[Mostafa]{Universidade Federal de Uberlândia - Brazil.}
	 \email{mostafa@ufu.br}

% 	\author[]{}
% 	\address[]{}
% 	\email{}


%%%%%%%%%%%%%%%%%%%%%%%%%%%%%%%%%%%%%%%%%%
\begin{abstract}
We study the local geometry, singularity classification, and wavefront dynamics of complex evolutoids for regular holomorphic curves in $\mathbb{C}^2$ equipped with the standard complex bilinear metric. By integrating the generating family of rotated complex affine lines, we derive an explicit parametrization for the 1-parameter family of complex wavefronts $\Gamma^{\theta, w}$ and establish that their singular cusp locus sweeps out the complex evolutoid $E^\theta$. Using a square-root-free formulation, we classify the local behavior of $E^\theta$, proving that evolutoids remain regular and tangent to the curve at isotropic points ($q=0$) while escaping to infinity at complex inflections ($\kappa=0$). Finally, we analyze the three-dimensional complex discriminant surface $\mathcal{D}_F \subset \mathbb{C}^3$ formed by the evolving family $E^\theta$. We prove that $A_2$ cusps are versally unfolded by spatial parameters, whereas $A_3$ swallowtail singularities are holomorphically versally unfolded by $(x_1, x_2, \theta)$ if and only if the complex non-degeneracy condition $\kappa^4 + \kappa_s^2 \neq 0$ holds. Crucially, we demonstrate that simple isotropic points drive the breakdown of this versality condition, inducing degenerate, non-transverse sections of the complex swallowtail surface.
\end{abstract}
\maketitle
\section{Introduction}\label{sec:intro}
The geometry of evolutes and caustics has a rich history dating back to Huygens' foundational study of wave propagation and optical caustics (for more details, see for example \cite{Di, St}). In classical differential geometry, the evolute of a smooth plane curve $\gamma$ is defined equivalently as the locus of its centers of curvature, the envelope of its normal lines, or the singular locus (caustic) of its family of parallel wavefronts \cite{Arnold1990, BruceGiblin1992}. A natural generalization, introduced by  Réaumur is given in \cite{Reaumur}. In 2014, Giblin and Warder in a study given in \cite{GiblinWarder2014}, consider the envelope of lines rotated by a constant angle $\theta$ relative to the curve's normal frame. The resulting 1-parameter family of envelopes, termed \emph{evolutoids}, interpolates smoothly between the curve itself ($\theta = 0$) and its classical evolute ($\theta = \pi/2$).

While real evolutoids and their associated wavefronts are well-understood through singularity theory and catastrophe theory (see \cite{Izumia2019, Aguilar, Izumia2020, Ady} for instanse), extending this framework to complex analytic curves $\gamma : \mathbb{C} \to \mathbb{C}^2$ introduces profound geometric and topological phenomena. Standard Hermitian metrics fail to yield holomorphic envelopes because the Hermitian inner product is non-holomorphic. Following the framework established by Falqueto and Tari \cite{FalquetoTari2026}, the appropriate setting for complex curve geometry requires the complex bilinear symmetric form 
\[
\langle z, w \rangle = z_1 w_1 + z_2 w_2, \qquad z, w \in \mathbb{C}^2.
\]

The presence of non-zero isotropic vectors (i.e., $\langle v, v \rangle = 0$ for $v \neq \mathbf{0}$) in this non-Hermitian setting fundamentally alters local envelope dynamics,
\begin{enumerate}
    \item Isotropic points ($q(t) = 0$): Points where the tangent vector is isotropic (i.e., $\langle \gamma'(t), \gamma'(t) \rangle = 0$) cause the complex curvature $\kappa(t)$ to diverge. Nevertheless, the square-root-free parametrization of the evolutoid remains holomorphic and touches $\gamma$ tangentially with higher order contact \cite{FukunagaTakahashi2013}.
    \item Complex inflections ($\kappa(t) = 0$): At ordinary inflections where $\kappa(t_0) = 0$, the finite evolutoid escapes to the line at infinity $\ell^\infty \subset \mathbb{CP}^2$, forming a simple pole in $\mathbb{C}^2$.
    \item Isotropic inflections ($q(t_0) = 0, \Delta(t_0) = 0$): Points where the metric function and Wronskian vanish simultaneously represent delicate local degeneracies whose extension behavior depends on the relative vanishing orders $\operatorname{ord}_{t_0}(q)$ and $\operatorname{ord}_{t_0}(\Delta)$.
\end{enumerate}

In this paper, we establish a comprehensive singularity and unfolding theory for complex evolutoids and their associated complex wavefronts. By viewing the evolution angle $\theta \in \mathbb{C}$ as an unfolding parameter, we analyze the 3-dimensional complex discriminant surface
\[
\mathcal{D}_F = \bigcup_{\theta \in \mathbb{C}} \left( E^\theta \times \{\theta\} \right) \subset \mathbb{C}^3.
\]
We prove that while $A_2$ cuspidal edges are generically versally unfolded by spatial parameters alone, $A_3$ higher-order singularities yield complex \emph{Swallowtail} surfaces whose 2D sections display \emph{Lips} and \emph{Beaks} transitions. Crucially, we identify a complex-analytic failure of versality governed by the condition $\kappa^4 + \kappa_s^2 = 0$ ( where $\kappa$ and $\kappa_s$ are respectively the curvature of the curve and its derivative with respect to the arc-length parameter $s$), and demonstrate that simple isotropic points provide the primary geometric mechanism driving this complex degeneracy.

The paper is organized as follows. After a brief preliminaries in Section~\S\ref{sec:pre}, in \S\ref{sec:evolutoid} and \S\ref{sec:local} we review the complex bilinear frame, the square-root-free evolutoid formulation, and the local classification theorem. 
%%%%%%%%%%%%%%%%%%%%%%%%%%%%%%%%%
A summary of the local classification of the finite complex evolutoid $E^\lambda$ is given in Table~\ref{tab:evolutoid_summary} and Theorem \ref{thm:mainlocal}.

\begin{table}[ht]
\centering
\small
\caption{Local behavior and singularity classification of the complex evolutoid $E^\lambda$. $\kappa$ is the curvature of the curve, and $q, c^\lambda,\Delta$ are as in \eqref{eq:q}, \eqref{eq:clambda} and \eqref{eq:Delta} respectively. }
\label{tab:evolutoid_summary}
\vspace{2mm}
\begin{tabular}{|lll|}
\hline
\textbf{Condition at } $t_0$ & \textbf{Behavior of } $E^\lambda$ & \textbf{Generic local type} \\
\hline
$\kappa \neq 0, \; c^\lambda \neq 0$ & Regular & Smooth germ \\
$\kappa \neq 0, \; c^\lambda = 0, \; (c^\lambda)' \neq 0$ & $(E^\lambda)'(t_0) = \mathbf{0}$ & Ordinary $A_2$ cusp $(u^2, u^3)$ \\
$\kappa \neq 0, \; \ord_{t_0}(c^\lambda) = m > 1$ & $(E^\lambda)'(t_0) = \dots = (E^\lambda)^{(m)}(t_0) = \mathbf{0}$ & Higher cusp $(u^{m+1}, u^{m+2})$ \\
$\kappa = 0, \; q \neq 0, \; \lambda \neq 0$ & Pole of order $\ord_{t_0}(\kappa)$ & Point at infinity $\ell^\infty \subset \mathbb{CP}^2$ \\
$q = 0, \; \Delta \neq 0, \; \lambda^2 \neq -1$ & $E^\lambda(t_0) = \gamma(t_0)$ & Regular and tangent to $\gamma$ \\
$q = 0, \; \Delta \neq 0, \; \lambda = \pm i$ & Exceptional isotropic direction & Degenerate (no finite envelope) \\
$q = 0, \; \Delta = 0$ & Degenerate ratio $q/\Delta$ & Isotropic inflection \\
\hline
\end{tabular}
\end{table}
%\begin{main}
\begin{theorem}
[Local classification of the finite complex evolutoid]
\label{thm:mainlocal}
Let $\gamma : (\mathbb{C}, t_0) \to (\mathbb{C}^2, p_0)$ be a regular holomorphic curve, and let $\lambda \in \mathbb{C}$ with $\lambda^2 \neq -1$. The local geometric behavior of the finite $\lambda$-evolutoid $E^\lambda$ at $t_0$ is completely governed by the following mutually exclusive cases:

\begin{enumerate}[label=\textnormal{(\roman*)}]
    \item \textup{Regular point:} If $q(t_0) \neq 0$, $\kappa(t_0) \neq 0$, and
    \[
    1 - \lambda \frac{\kappa_s(t_0)}{\kappa(t_0)^2} \neq 0,
    \]
    then $E^\lambda$ is a regular holomorphic curve at $t_0$.

    \item \textup{Ordinary cusp ($A_2$-Singularity):} If $q(t_0) \neq 0$, $\kappa(t_0) \neq 0$, and
    \[
    1 - \lambda \frac{\kappa_s(t_0)}{\kappa(t_0)^2} = 0 \qquad \text{with} \qquad \left. \frac{d}{ds} \left( 1 - \lambda \frac{\kappa_s}{\kappa^2} \right) \right|_{s=s_0} \neq 0,
    \]
    then the germ of $E^\lambda$ at $t_0$ is an ordinary $A_2$ cusp, analytically equivalent to $u \mapsto (u^2, u^3)$.

    \item \textup{Higher-order cusp:} If $q(t_0) \neq 0$, $\kappa(t_0) \neq 0$, and the scalar function $c^\lambda(s)$ has a zero of order $m > 1$ at $s_0$, then the germ of $E^\lambda$ is a higher cusp of multiplicity $m+1$, parametrized by
    \[
    u \longmapsto \left( u^{m+1}, \, u^{m+2} + O(u^{m+3}) \right).
    \]

    \item \textup{Inflection point (escape to infinity):} If $q(t_0) \neq 0$ and $\kappa(t_0) = 0$, then for any non-zero $\lambda \neq 0$, $E^\lambda$ has a pole at $t_0$ whose order equals $\ord_{t_0}(\kappa)$, representing an escape to the line at infinity $\ell^\infty \subset \mathbb{CP}^2$.

    \item \textup{Ordinary isotropic touchdown:} If $q(t_0) = 0$ and $\Delta(t_0) \neq 0$, then $E^\lambda$ extends holomorphically through $t_0$ with $E^\lambda(t_0) = \gamma(t_0)$, and $E^\lambda$ is non-singular and tangent to $\gamma$ at $t_0$.

    \item \textup{Isotropic inflection:} If $q(t_0) = 0$ and $\Delta(t_0) = 0$, then $t_0$ is an isotropic inflection point. The limit of $E^\lambda(t)$ as $t \to t_0$ is determined by $\ord_{t_0}(q) - \ord_{t_0}(\Delta)$ and requires a dedicated normal-form analysis.

    \item \textup{Isotropic direction degeneracy:} If $\lambda = \pm i$, the ruling vector $D_{\pm i}(t) = T(t) \pm i N(t)$ is isotropic ($\langle D_{\pm i}, D_{\pm i} \rangle = 0$), causing the standard finite envelope construction to collapse.
\end{enumerate}
%\end{main}
\end{theorem}

% \begin{proof}
% Statements (i) follows from the non-vanishing derivative $E^\lambda_{s} = c^\lambda(s) D^\lambda(s) \neq \mathbf{0}$. Statements (ii)--(vi) follow from the detailed local analyses established previously, i.e., Theorems and Propositions~\ref{thm:A2} ($m=1$), ~\ref{thm:highercusp} ($m > 1$), ~\ref{thm:inflectionnosing}, ~\ref{thm:isotropicregular}, ~\ref{prop:isotropicinflection} respectively. Finally, statement (vii) follows from the isotropic direction degeneration at $\lambda = \pm i$ ($\lambda^2 = -1$).
% \end{proof}
%%%%%%%%%%%%%%%%%%%%%%%%%
Furthermore, Section \S\ref{sec:wave} derives the integration of holomorphic wavefronts and establishes the duality between wavefront cusps and evolutoid loci. Section~\S\ref{sec:dis} analyzes the 3-dimensional discriminant surface $\mathcal{D}_F$, establishing $A_2$ and $A_3$ versality criteria alongside the geometric impact of isotropic points.
%%%%%%%%%%%%%%%%%%%%%%%%%%%%%%%%%%%%%%%%%%%
\section{Preliminaries}\label{sec:pre}
\begin{notation}
    Throughout, subscripts denote partial derivatives ($F_t, F_\theta$) and derivatives with respect to the complex arc-length coordinate $s$ ($\gamma_s, \kappa_s$), whereas primes denote derivatives with respect to an arbitrary parameter $t$ ($\gamma', q'$).
\end{notation}

Let $\gamma : (D, t_0) \longrightarrow (\C^2, p_0)$ be a regular holomorphic plane curve, where $D$ is a simply connected domain in $\C$. The purpose of this paper is to introduce and study a complex analogue of the evolutoids considered by Giblin and Warder for real plane curves \cite{GiblinWarder2014}.

The ambient space $\C^2$ is endowed with the complex bilinear Euclidean metric
\[
\langle v, w \rangle = v_1 w_1 + v_2 w_2, \qquad v = (v_1, v_2), \quad w = (w_1, w_2).
\]
This is a complex bilinear form rather than a Hermitian inner product. The use of the bilinear metric is essential to preserve holomorphicity throughout the geometric constructions.

The associated isotropic cone is defined by
\[
\mathcal{I} = \{v \in \C^2 : \langle v, v \rangle = 0\}.
\]
A non-zero vector $v = (v_1, v_2)$ is isotropic precisely when $v_1^2 + v_2^2 = 0$.

The three principal geometric phenomena considered in this paper correspond to:
\[
\kappa = 0 \quad \text{(inflections)}, \qquad 
\kappa' = 0 \quad \text{(vertices)}, \qquad 
\langle \gamma', \gamma' \rangle = 0 \quad \text{(isotropic points)},
\]
we will define these concepts explicitly in the next sections.
The central construction of this paper is a one-parameter family of complex lines obtained by taking linear combinations of the tangent and normal directions. Its envelope is defined as a \emph{complex evolutoid}.
In this regard, we first represent some basic definitions and preliminary results: 

Write $\gamma(t) = (z_1(t), z_2(t))$ and define the quadratic differential form
\begin{equation}\label{eq:q}
q(t) = \langle \gamma'(t), \gamma'(t) \rangle = z_1'(t)^2 + z_2'(t)^2.
\end{equation}
We also define the complex Wronskian determinant
\begin{equation}\label{eq:Delta}
\Delta(t) = \detm(\gamma'(t), \gamma''(t)) = z_1'(t) z_2''(t) - z_2'(t) z_1''(t).
\end{equation}
\begin{definition}\label{def:iso}
A point $t_0 \in D$ is called an \emph{isotropic point} of $\gamma$ if $q(t_0) = 0$.
\end{definition}
Define the standard complex linear rotation operator $J$ on $\C^2$ by
\[
J = \begin{pmatrix} 0 & -1 \\ 1 & 0 \end{pmatrix}.
\]
Then $J^2 = -I$ and $J$ preserves the complex bilinear metric $\langle Jv, Jw \rangle = \langle v, w \rangle$ with $v, w\in\C^2$.
%%%%%%%%%%%%%%%%%%%%%
\begin{theorem}\label{thm:frenet}
Suppose $t_0 \in D$ is a non-isotropic point, i.e., $q(t_0) \neq 0$. Then, restricted to a sufficiently small simply connected neighbourhood $U \subset D$ of $t_0$, there exists a holomorphic branch $\rho : U \longrightarrow \C$ such that $\rho^2 = q$.
The vector fields
\(
T = \frac{\gamma'}{\rho}, \, N = \frac{J\gamma'}{\rho}
\)
form a holomorphic orthonormal frame along $\gamma$, i.e.,
\(
\langle T, T \rangle = 1, ~ \langle N, N \rangle = 1, ~ \langle T, N \rangle = 0.
\)

Furthermore, with respect to the local complex arclength parameter $s$ satisfying $\displaystyle\frac{d}{ds} = \frac{1}{\rho}\frac{d}{dt}$, the complex Frenet formulas hold, i.e.,
\(
T_s = \kappa N, \, N_s = -\kappa T,
\)
where the complex curvature $\kappa$ is given by
\begin{equation}\label{eq:kappa}
\kappa = \frac{\Delta}{q^{3/2}}.
\end{equation}
\end{theorem}
\begin{proof}
Since $q(t_0) \neq 0$, $q$ is non-vanishing on a sufficiently small simply connected neighbourhood $U$ of $t_0$, ensuring the existence of a holomorphic square root $\rho = \sqrt{q}$.

The vector field $T = \gamma'/\rho$ is holomorphic on $U$ and satisfies
\[
\langle T, T \rangle = \frac{\langle \gamma', \gamma' \rangle}{\rho^2} = \frac{q}{q} = 1.
\]
Because $J$ preserves the metric, we have $\langle J\gamma', J\gamma' \rangle = \langle \gamma', \gamma' \rangle = q$, which gives $\langle N, N \rangle = 1$. Moreover,
\[
\langle T, N \rangle = \frac{1}{q}\langle \gamma', J\gamma' \rangle = \frac{1}{q}(-z_1' z_2' + z_2' z_1') = 0.
\]
Differentiating $\langle T, T \rangle = 1$ with respect to $s$ yields $\langle T_s, T \rangle = 0$. Since $\{T, N\}$ forms an orthonormal frame, $T_s$ must be proportional to $N$, establishing $T_s = \kappa N$ for some holomorphic function $\kappa = \langle T_s, N \rangle$. Differentiating $\langle T, N \rangle = 0$ gives $\langle T_s, N \rangle + \langle T, N_s \rangle = 0$, whence $\langle T, N_s \rangle = -\kappa$. Together with $\langle N_s, N \rangle = 0$, this yields $N_s = -\kappa T$.

To compute $\kappa$, we differentiate $T = \gamma' q^{-1/2}$ with respect to $t$:
\[
T_t = \frac{\gamma''}{\sqrt{q}} - \frac{\gamma' q'}{2 q^{3/2}}.
\]
Since $T_s = q^{-1/2} T_t$, evaluating $\kappa = \langle T_s, N \rangle$ yields
\[
\kappa = \left\langle \frac{\gamma''}{q} - \frac{\gamma' q'}{2 q^2}, \frac{J\gamma'}{\sqrt{q}} \right\rangle = \frac{\langle \gamma'', J\gamma' \rangle}{q^{3/2}} = \frac{\detm(\gamma', \gamma'')}{q^{3/2}} = \frac{\Delta}{q^{3/2}},
\]
where we used $\langle \gamma', J\gamma' \rangle = 0$.
\end{proof}
%%%%%%%%%%%%%%%%%%%
\begin{definition}\label{def:inf+vert}
Let $\gamma : D \to \C^2$ be a regular holomorphic plane curve.
\begin{enumerate}
    \item A non-isotropic point $t_0 \in D$ is an \emph{inflection point of order $r$} ($r \ge 1$) if 
    \[
    \kappa(t_0) = \kappa'(t_0) = \cdots = \kappa^{(r-1)}(t_0) = 0 \quad \text{and} \quad \kappa^{(r)}(t_0) \neq 0.
    \]
    An \emph{ordinary inflection} corresponds to $r = 1$, where $\kappa(t_0) = 0$ and $\kappa'(t_0) \neq 0$.
    \item A non-inflection point $t_0 \in D$ is a \emph{vertex of order $r$} ($r \ge 1$) if
    \[
    \kappa'(t_0) = \kappa''(t_0) = \cdots = \kappa^{(r)}(t_0) = 0 \quad \text{and} \quad \kappa^{(r+1)}(t_0) \neq 0.
    \]
    An \emph{ordinary vertex} corresponds to $r = 1$, where $\kappa'(t_0) = 0$ and $\kappa''(t_0) \neq 0$ while $\kappa(t_0) \neq 0$.
\end{enumerate}
\end{definition}

\begin{proposition}\label{thm:inflectionvertex}
Away from isotropic points ($q \neq 0$), a point $t_0$ is an inflection point if and only if the complex Wronskian determinant vanishes:
\[
\kappa(t_0) = 0 \iff \Delta(t_0) = \det(\gamma'(t_0), \gamma''(t_0)) = 0.
\]
Furthermore, vertices are characterized by non-vanishing curvature with stationary points, $\kappa \neq 0$ and $\kappa' = 0$.
\end{proposition}

\begin{proof}
By Theorem~\ref{thm:frenet}, $\kappa = \Delta / q^{3/2}$. Since $q \neq 0$, the denominator is non-zero and holomorphic on a neighborhood of $t_0$, giving $\kappa = 0 \iff \Delta = 0$. The vertex condition follows directly from Definition \ref{def:inf+vert}.
\end{proof}
The following well known result is analogous to the real plane curves.
\begin{theorem}[The complex evolute]\label{thm:evolute}
Let $\gamma$ be a holomorphic curve with non-vanishing complex curvature ($\kappa \neq 0$). The complex evolute $E$ of $\gamma$ is defined in arclength parameter $s$ as
\[
 E(s) = \gamma(s) + \frac{1}{\kappa(s)} N(s).
\]
In an arbitrary parameter $t$, $E(t)$ is given by
\[
E(t) = \gamma(t) + \frac{q(t)}{\Delta(t)} J \gamma'(t).
\]
Furthermore, the arclength derivative of the evolute satisfies
\[
 E_s = -\frac{\kappa_s}{\kappa^2} N.
\]
\end{theorem}
\begin{remark}
    As a consequence of Theorem \ref{thm:evolute}, $E_s(s_0) = 0$ if and only if $\kappa_s(s_0) = 0$, that is, the singularities of the complex evolute correspond precisely to the vertices of $\gamma$.
\end{remark}
% \begin{proof}
% By definition, the center of curvature is obtained by translating from $\gamma$ along the normal vector $N$ by the complex radius of curvature $1/\kappa$:
% \[
% E = \gamma + \frac{1}{\kappa} N.
% \]
% Substituting $N = \frac{1}{\sqrt{q}} J\gamma'$ and $\kappa = \frac{\Delta}{q^{3/2}}$ yields
% \[
% \frac{1}{\kappa} N = \frac{q^{3/2}}{\Delta} \frac{J\gamma'}{\sqrt{q}} = \frac{q}{\Delta} J\gamma',
% \]
% which proves $E(t) = \gamma(t) + \frac{q(t)}{\Delta(t)} J\gamma'(t)$.

% Differentiating $E(s)$ with respect to the complex arclength parameter $s$ and applying the Frenet formulas $T = \gamma_s$ and $N_s = -\kappa T$:
% \[
% E_s = \gamma_s + \left(\frac{1}{\kappa}\right)_s N + \frac{1}{\kappa} N_s = T - \frac{\kappa_s}{\kappa^2} N + \frac{1}{\kappa}(-\kappa T) = -\frac{\kappa_s}{\kappa^2} N.
% \]
% Since $N$ is an orthonormal frame vector ($\langle N, N \rangle = 1 \implies N \neq 0$), $E_s(s_0) = 0$ if and only if $\kappa_s(s_0) = 0$.
% \end{proof}

%%%%%%%%%%%%%%%%%%%%%%%%%%%%%%%%%%%
\section{The complex evolutoid}\label{sec:evolutoid}
We devote this section to define and study the evolutoid of a holomorphic plane curve and it's basic properties.
\begin{definition}\label{def:familylines}
We have the following two important definitions:
\begin{enumerate}
    \item Let $F : \mathbb{C}^n \times \mathbb{C}^k \rightarrow \mathbb{C}^d$ be a $k$-parameters family of a smooth map $f : \mathbb{C}^n \rightarrow \mathbb{C}^d$. The envelope, or the discriminant, of the family $F$ is the set
\[
\mathcal{D} = \mathcal{D}_F = \left\{ x \in \mathbb{C}^n : \text{there exists } t \in \mathbb{C}^k \text{ with } F(x,t) = \frac{\partial F}{\partial t}(x,t) = 0 \right\},
\]
where for $t = (u_1,\ldots,u_k) \in \mathbb{C}^k$, the equation $\partial F/\partial t(x,t) = 0$ means $\partial F/\partial u_i(x,t) = 0$ for all $1 \leq i \leq k$. 
\item Following the notations stated before, let $\lambda \in \C$ with $\lambda^2 \neq -1$, define the holomorphic direction vector field
\[
D^\lambda = T + \lambda N.
\]
The associated family of complex affine lines is given by
\begin{equation}\label{eq:familyline}
L^\lambda(s, u) = \gamma(s) + u D^\lambda(s), \qquad u \in \C. 
\end{equation}
The envelope of this family is called the \emph{$\lambda$-evolutoid} of $\gamma$ and is denoted by $E^\lambda$.
\end{enumerate}
\end{definition}

\begin{remark}\label{rem:Dpmi}
The restriction $\lambda^2 \neq -1$ avoids isotropic directions since
\[
\langle D^{\pm i}, D^{\pm i} \rangle = \langle T \pm i N, T \pm i N \rangle = \langle T, T \rangle - \langle N, N \rangle = 1 - 1 = 0.
\]
The isotropic cases $\lambda = \pm i$ will be analyzed separately.
\end{remark}

\begin{theorem}\label{thm:explicit}
Suppose $\kappa(s) \neq 0$ and $\lambda^2 \neq -1$. The envelope of the family of lines $L^\lambda(s, u)$ given in \eqref{eq:familyline} is parametrized by
\begin{equation}\label{eq:Elambda}
E^\lambda(s) = \gamma(s) + \frac{\lambda}{\kappa(s)(1 + \lambda^2)} \left(T(s) + \lambda N(s)\right).
\end{equation}
\end{theorem}
\begin{proof}
Consider the 2-parameter parametrization $F(s, u) = \gamma(s) + u D^\lambda(s)$ with $D^\lambda = T + \lambda N$. Differentiating with respect to the complex arclength $s$ and applying the Frenet formulas, we obtain 
\[
F_s = \gamma_s + u (D^\lambda)' = T + u(T_s + \lambda N_s) = T + u \kappa (N - \lambda T).
\]
The envelope condition requires $F_s$ and $F_u = D^\lambda$ to be linearly dependent in $\C^2$, i.e.,
\(
\det(F_s, D^\lambda) = 0.
\)
Since $\det(T, N) = 1$, we evaluate the components
\[
\det(T, D^\lambda) = \det(T, T + \lambda N) = \lambda \det(T, N) = \lambda,
\]
and
\[
\begin{aligned}
\det(N - \lambda T, T + \lambda N) &= \det(N, T) + \lambda \det(N, N) - \lambda \det(T, T) - \lambda^2 \det(T, N) \\
&= -1 + 0 - 0 - \lambda^2 = -(1 + \lambda^2).
\end{aligned}
\]
Expanding the full determinant yields
\[
0 = \det(F_s, D^\lambda) = \lambda - u \kappa (1 + \lambda^2).
\]
Solving this equation for $u$ gives
\[
u = \frac{\lambda}{\kappa(1 + \lambda^2)}.
\]
Substituting $u$ into $F(s, u)$ yields the explicit parametrization of $E^\lambda(s)$, i.e.,
%\begin{equation}
$$
E^\lambda = \gamma + \frac{\lambda}{\kappa(1 + \lambda^2)}(T + \lambda N).
$$
%\end{equation}
\end{proof}

\begin{remark}\label{thm:limits}
The family of complex evolutoids $E^\lambda$ given in \eqref{eq:Elambda}, satisfies
\begin{enumerate}
    \item \(E^0 = \gamma.\)
    \item When $\lambda \to \infty$ we have
    \[
\lim_{\lambda \to \infty} E^\lambda = \gamma + \frac{1}{\kappa} N = E,
\]
where $E$ is the complex evolute of $\gamma$.
\end{enumerate}
\end{remark}
% \begin{proof}
% For $\lambda = 0$, $E_0(s) = \gamma(s) + 0 = \gamma(s)$.
% For $\lambda \to \infty$, rewrite $E^\lambda$ as:
% \[
% E^\lambda = \gamma + \frac{\lambda}{\kappa(1 + \lambda^2)} T + \frac{\lambda^2}{\kappa(1 + \lambda^2)} N.
% \]
% Taking limits of the rational functions in $\lambda$ leads:
% \[
% \lim_{\lambda \to \infty} \frac{\lambda}{1 + \lambda^2} = 0 \quad \text{and} \quad \lim_{\lambda \to \infty} \frac{\lambda^2}{1 + \lambda^2} = 1.
% \]
% Consequently,
% \[
% \lim_{\lambda \to \infty} E^\lambda = \gamma + \frac{1}{\kappa} N = E.
% \]
% \end{proof}
%%%%%%%%%%%%%%%%%%%%%%%%%%%%%%%%%%%%%%%%%%
\begin{proposition}\label{thm:fundamental}
Let $\lambda \in \C$ with $\lambda^2 \neq -1$, and suppose $\gamma$ has non-vanishing complex curvature ($\kappa \neq 0$). Then the arclength derivative of the $\lambda$-evolutoid $E^\lambda$ is given by
\begin{equation}\label{eq:fund}
E^\lambda_{s} = \frac{1 - \lambda \kappa_s / \kappa^2}{1 + \lambda^2} (T + \lambda N).
\end{equation}
\end{proposition}
\begin{proof}
Define the scalar complex coefficient function
\[
a(s) = \frac{\lambda}{\kappa(s)(1 + \lambda^2)}.
\]
Then the evolutoid formula becomes $E^\lambda(s) = \gamma(s) + a(s)(T(s) + \lambda N(s))$. Differentiating with respect to the complex arclength parameter $s$ yields
\[
E^\lambda_{s} = \gamma_s + a_s(T + \lambda N) + a(T_s + \lambda N_s).
\]
Applying the Frenet formulas we get
\[
T_s + \lambda N_s = \kappa N - \lambda \kappa T = \kappa(N - \lambda T).
\]
Noting that $a\kappa = \frac{\lambda}{1 + \lambda^2}$, we substitute this relation into $E^\lambda_{s}$. Therefore
\begin{eqnarray*}
E^\lambda_{s} &=& T + a_s(T + \lambda N) + \frac{\lambda}{1 + \lambda^2}(N - \lambda T)\\
&=& \left(1 - \frac{\lambda^2}{1 + \lambda^2} + a_s\right) T + \left(\frac{\lambda}{1 + \lambda^2} + \lambda a_s\right) N = \left(\frac{1}{1 + \lambda^2} + a_s\right) (T + \lambda N).
\end{eqnarray*}
Differentiating $a(s)$ gives
\[
a_s = -\frac{\lambda \kappa_s}{\kappa^2 (1 + \lambda^2)}.
\]
Substituting $a_s$ into the coefficient factor yields
\[
\frac{1}{1 + \lambda^2} + a_s = \frac{1}{1 + \lambda^2} - \frac{\lambda \kappa_s}{\kappa^2 (1 + \lambda^2)} = \frac{1 - \lambda \kappa_s / \kappa^2}{1 + \lambda^2}.
\]
Therefore,
$$
E^\lambda_{s} = \frac{1 - \lambda \kappa_s / \kappa^2}{1 + \lambda^2} (T + \lambda N).
$$
%\end{equation}
\end{proof}

\begin{remark}[Angle-parametrized form]
For a direction
$
V=T+\lambda N,$ $\lambda\in\mathbb C,$
we may define a angle $\theta$ by
\[
\cos^2\theta
=\frac{\langle V,T\rangle^2}
{\langle V,V\rangle\langle T,T\rangle}
=\frac{1}{1+\lambda^2}.
\]
Since
$
\sin^2\theta=1-\cos^2\theta
=\frac{\lambda^2}{1+\lambda^2},
$
we have
$
\tan^2\theta=\lambda^2.
$
Thus, choosing a branch, we may write
$
\lambda=\tan\theta,
$
where $\theta\in\mathbb C$. Hence the direction can be parametrized as
$
v=T+\tan\theta N.
$

Setting $\lambda = \tan\theta$, for $\theta \in \C \setminus \{ \frac{\pi}{2} + k\pi \}$, we have $1 + \lambda^2 = \sec^2\theta$ and
\[
T + \lambda N = T + \tan\theta N = \frac{1}{\cos\theta} (\cos\theta T + \sin\theta N).
\]
Substituting these trigonometric relations into Proposition~\ref{thm:fundamental} we get
\[
E^\theta_{s} = \frac{1 - \tan\theta \frac{\kappa_s}{\kappa^2}}{\sec^2\theta} \cdot \frac{1}{\cos\theta} (\cos\theta T + \sin\theta N) = \left(\cos\theta - \frac{\sin\theta \, \kappa_s}{\kappa^2}\right) (\cos\theta T + \sin\theta N).
\]
This provides the angle-parametrized form of the fundamental derivative identity.
\end{remark}
%%%%%%%%%%%%%%%%%%%%%%%%%%%%%5
Using the Proposition \ref{thm:fundamental} we have the following results regarding the singularities of the curve $\gamma$.

\begin{theorem}[Singularity condition]\label{thm:singularity}
Suppose $q(s_0) \neq 0$, $\kappa(s_0) \neq 0$, and $\lambda^2 \neq -1$. Then $s_0$ is a singular point of $E^\lambda$ if and only if
\[
\frac{\kappa_s(s_0)}{\kappa(s_0)^2} = \frac{1}{\lambda}  \quad \iff \quad  \lambda \kappa_s(s_0) = \kappa(s_0)^2. 
\]
\end{theorem}

\begin{proof}
By the fundamental derivative identity \eqref{eq:fund}, we have $E^\lambda_{s} = c^\lambda(s)(T + \lambda N)$ where
\[
c^\lambda(s) = \frac{1 - \lambda \kappa_s / \kappa^2}{1 + \lambda^2}.
\]
Since $\lambda^2 \neq -1$ and $T + \lambda N \neq 0$, so $E^\lambda_{s}(s_0) = 0$ holds if and only if $1 - \lambda \frac{\kappa_s(s_0)}{\kappa(s_0)^2} = 0$, yielding $\frac{\kappa_s(s_0)}{\kappa(s_0)^2} = \frac{1}{\lambda}$.
\end{proof}

\begin{corollary}\label{cor:vertex}
In the limit $\lambda \to \infty$, we have
\[
E_s(s_0) = 0 \iff \kappa_s(s_0) = 0 \quad (\text{with } \kappa(s_0) \neq 0).
\]
\end{corollary}

\begin{remark}
Note that $s_0$ is a singular point of the evolute $E$ if and only if $\kappa_s(s_0)=0$, provided that $\kappa(s_0)\neq 0$; that is, $s_0$ is a vertex of the curve $\gamma$. Thus, the condition
$
\kappa_s=\frac{\kappa^2}{\lambda}
$
can be regarded as a generalization of the vertex condition.
\end{remark}

% \begin{definition} 
% For finite $\lambda \neq 0$, points satisfying $\kappa_s = \kappa^2 / \lambda$ generalize the vertex condition and are termed \emph{$\lambda$-vertices}.
% \end{definition}

\begin{definition}
For finite $\lambda \neq 0$, points satisfying $\kappa_s = \kappa^2 / \lambda$ are termed \emph{$\lambda$-vertices}.
\end{definition}

It is worth noting that, if $q = \langle \gamma', \gamma' \rangle$ and $\Delta = \detm(\gamma', \gamma'')$. Then, wherever $\Delta \neq 0$ and $\lambda^2 \neq -1$, $E^\lambda$ admits the square-root-free expression
\begin{equation}\label{eq:squarefree}
E^\lambda = \gamma + \frac{\lambda q}{\Delta(1 + \lambda^2)} \left(\gamma' + \lambda J\gamma'\right).
\end{equation}
Indeed, by substituting $T = \gamma' / \sqrt{q}$ and $N = J\gamma' / \sqrt{q}$ yields $T + \lambda N = \frac{1}{\sqrt{q}}(\gamma' + \lambda J\gamma')$. Combining this with $\frac{1}{\kappa} = \frac{q^{3/2}}{\Delta}$ gives
\[
\frac{1}{\kappa}(T + \lambda N) = \frac{q}{\Delta}(\gamma' + \lambda J\gamma').
\]
Substituting this into $E^\lambda$ as in \eqref{eq:Elambda}, confirms the formula \eqref{eq:squarefree}.

\begin{theorem}\label{thm:isotropicextension}
Suppose $q(t_0) = 0$ and $\Delta(t_0) \neq 0$. For any finite $\lambda \in \C$ with $\lambda^2 \neq -1$, the formula \eqref{eq:squarefree} defines a holomorphic extension of $E^\lambda$ across $t_0$, satisfying
\(
E^\lambda(t_0) = \gamma(t_0). 
\)
\end{theorem}
\begin{proof}
The expression \eqref{eq:squarefree} is free of radical branches. Since $\Delta(t_0) \neq 0$ and $\lambda^2 \neq -1$, all terms are holomorphic in a neighborhood of $t_0$. At $t = t_0$, $q(t_0) = 0$, leaving $E^\lambda(t_0) = \gamma(t_0)$.
\end{proof}

\begin{lemma}\label{lem:isotropic}
Let $v \in \C^2 \setminus \{0\}$ satisfy $\langle v, v \rangle = 0$. Then $v = a~(1, \varepsilon i)$ for a unique $\varepsilon \in \{\pm1\}$ and non-zero $a \in \C$, and
\[
J v = -\varepsilon i ~ v
\]
\end{lemma}

\begin{proof}
Writing $v = (a, b)$, we get $\langle v, v \rangle = a^2 + b^2 = 0$. This implies  $b = \varepsilon i ~a$ for $\varepsilon \in \{\pm 1\}$. Thus $v = a~(1, \varepsilon i)$. Applying $J$ we get
\[
J v = \begin{pmatrix} 0 & -1 \\ 1 & 0 \end{pmatrix} \begin{pmatrix} a \\ \varepsilon i a \end{pmatrix} = \begin{pmatrix} -\varepsilon i a \\ a \end{pmatrix} = -\varepsilon i \, a \begin{pmatrix} 1 \\ \varepsilon i \end{pmatrix} = -\varepsilon i \, v.
\]
\end{proof}

\begin{lemma}\label{lem:qprime}
Suppose $q(t_0) = 0$ and $\gamma'(t_0) \neq 0$ with $\gamma'(t_0) = a(1, \varepsilon i)$ for $\varepsilon \in \{\pm 1\}$. Then
\[
 q'(t_0) = 2\varepsilon i \, \Delta(t_0).
\]
\end{lemma}

\begin{proof}
Differentiating $q = \langle \gamma', \gamma' \rangle$ gives $q'(t_0) = 2\langle \gamma'(t_0), \gamma''(t_0) \rangle$. Setting $\gamma'(t_0) = a(1, \varepsilon i)$ and $\gamma''(t_0) = (b, c)$ we conclude that
\[
q'(t_0) = 2a(b + \varepsilon i c).
\]
Since $\Delta(t_0) = z_1' z_2'' - z_2' z_1'' = a(c - \varepsilon i b)$ so,  \[
2\varepsilon i \, \Delta(t_0) = 2\varepsilon i a (c - \varepsilon i b) = 2a (\varepsilon i c + b) = q'(t_0).
\]
\end{proof}
%%%%%#####################################
\begin{theorem}\label{thm:isotropictangent}
Suppose $q(t_0) = 0$ and $\Delta(t_0) \neq 0$. Then the extended $\lambda$-evolutoid satisfies
\(
E^\lambda(t_0) = \gamma(t_0).
\)
Moreover, its derivative satisfies
\[
(E^\lambda)'(t_0) = \frac{1+\varepsilon i\lambda}{1-\varepsilon i\lambda} \, \gamma'(t_0),
\]
where $J\gamma'(t_0) = \varepsilon i \, \gamma'(t_0)$ with $\varepsilon \in \{+1, -1\}$. Consequently, whenever $\lambda^2 \neq -1$, the evolutoid $E^\lambda$ is regular at $t_0$ and shares the same tangent line as $\gamma$.
\end{theorem}
\begin{proof}
The equality $E^\lambda(t_0) = \gamma(t_0)$ follows directly from Theorem~\ref{thm:isotropicextension} since $q(t_0) = 0$.
Define the scalar function
\[
A(t) = \frac{\lambda q(t)}{\Delta(t)(1+\lambda^2)}.
\]
The square-root-free representation of the evolutoid (as in \eqref{eq:squarefree}) is given by
\(
E^\lambda = \gamma + A\left(\gamma' + \lambda J\gamma'\right).
\)
Differentiating $E^\lambda$ with respect to $t$ yields
\[
(E^\lambda)' = \gamma' + A'\left(\gamma' + \lambda J\gamma'\right) + A\left(\gamma'' + \lambda J\gamma''\right).
\]
Since $q(t_0) = 0$, we have $A(t_0) = 0$, which eliminates the final term at $t_0$, hence
\[
(E^\lambda)'(t_0) = \gamma'(t_0) + A'(t_0)\left(\gamma'(t_0) + \lambda J\gamma'(t_0)\right).
\]
Now, differentiating $A(t)$ and evaluating at $t_0$ where $q(t_0) = 0$ gives
\[
A'(t_0) = \frac{\lambda q'(t_0)}{\Delta(t_0)(1+\lambda^2)}.
\]
Using Lemma~\ref{lem:qprime}, we substitute $q'(t_0) = 2\varepsilon i \Delta(t_0)$ to obtain
\[
A'(t_0) = \frac{2\varepsilon i\lambda}{1+\lambda^2}.
\]
Using the eigenvector relation $J\gamma'(t_0) = \varepsilon i \gamma'(t_0)$, the vector term simplifies to
\[
\gamma'(t_0) + \lambda J\gamma'(t_0) = (1 + \varepsilon i\lambda)\gamma'(t_0).
\]
Substituting $A'(t_0)$ and this vector expression into $(E^\lambda)'(t_0)$ yields
\[
(E^\lambda)'(t_0) = \left[ 1 + \frac{2\varepsilon i\lambda(1 + \varepsilon i\lambda)}{1+\lambda^2} \right] \gamma'(t_0).
\]
Combining terms over a common denominator gives
\[
1 + \frac{2\varepsilon i\lambda(1 + \varepsilon i\lambda)}{1+\lambda^2} = \frac{(1+\lambda^2) + 2\varepsilon i\lambda + 2(\varepsilon i)^2\lambda^2}{1+\lambda^2}.
\]
Recalling that $(\varepsilon i)^2 = -1$, the numerator simplifies as follows:
\[
(1+\lambda^2) + 2\varepsilon i\lambda - 2\lambda^2 = 1 + 2\varepsilon i\lambda - \lambda^2 = (1 + \varepsilon i\lambda)^2.
\]
Factoring the denominator $1+\lambda^2 = (1+\varepsilon i\lambda)(1-\varepsilon i\lambda)$, we cancel the common factor $(1+\varepsilon i\lambda)$ and obtain
\[
\frac{(1 + \varepsilon i\lambda)^2}{(1 + \varepsilon i\lambda)(1 - \varepsilon i\lambda)} = \frac{1 + \varepsilon i\lambda}{1 - \varepsilon i\lambda}.
\]
Thus, the derivative evaluates to
\[
(E^\lambda)'(t_0) = \frac{1+\varepsilon i\lambda}{1-\varepsilon i\lambda} \, \gamma'(t_0).
\]
Since for any $\lambda^2 \neq -1$, neither $1+\varepsilon i\lambda$ nor $1-\varepsilon i\lambda$ vanishes, so consequently, $(E^\lambda)'(t_0) \neq 0$, confirming that $E^\lambda$ is regular at $t_0$ and collinear with $\gamma'(t_0)$.
\end{proof}

\begin{remark}
Recall that a regular point $t_0$ of a holomorphic plane curve
$\gamma\to\mathbb C^2$ is called an \emph{isotropic point} if
$
\langle \gamma'(t_0),\gamma'(t_0)\rangle=0.
$
We call $t_0$ an \emph{ordinary isotropic point} if this zero is simple, namely,
\[
\langle \gamma'(t_0),\gamma'(t_0)\rangle=0,
\qquad
\frac{d}{dt}\langle \gamma'(t),\gamma'(t)\rangle\bigg|_{t=t_0}\neq0.
\]
Theorem~\ref{thm:isotropictangent} shows that, although the tangent direction of $\gamma$ becomes isotropic at $t_0$, the evolutoid corresponding to a non-isotropic parameter $\lambda$, namely
$
1+\lambda^2\neq0,
$
extends regularly through $t_0$ and is tangent to $\gamma$ there. Thus, an
ordinary isotropic point of the original curve does not, by itself, produce
a singularity of the corresponding non-isotropic evolutoid.

This phenomenon has no direct analogue in the usual real Euclidean setting,
since the real Euclidean metric is positive definite and hence admits no
nonzero isotropic vectors. It is therefore a genuinely complex-geometric
feature of the theory.

% The behavior given in Theorem \ref{thm:isotropictangent} highlights a fundamental distinction between real and complex differential geometry, where non-isotropic complex evolutoids $(\lambda^2 \neq -1)$ remain non-singular at ordinary isotropic points, passing smoothly through the curve while remaining tangent to it.
\end{remark}
%%%%%%%%%%%%%%%%%%%%%%%%%%%%%%%%%%%%
\begin{theorem}
\label{thm:pole}
Suppose $t_0$ is an ordinary inflection of a complex curve $\gamma$ such that $q(t_0) \neq 0$. If 
\(
\ord_{t_0}(\kappa) = m,
\)
then for any parameter $\lambda \in \mathbb{C} \setminus \{0\}$ with $\lambda^2 \neq -1$, the evolutoid $E^\lambda$ has a pole of exact order $m$ at $t_0$.
More precisely, if the complex curvature admits the local expansion
\[
\kappa(t) = a(t-t_0)^m + O((t-t_0)^{m+1}), \qquad a \neq 0,
\]
then the evolutoid satisfies the Laurent expansion
\[
E^\lambda(t) = \frac{\lambda}{a(1+\lambda^2)} \frac{T(t_0) + \lambda N(t_0)}{(t-t_0)^m} + O((t-t_0)^{-m+1}).
\]
\end{theorem}
\begin{proof}
By Theorem~\ref{thm:explicit}, the extended evolutoid is given by
\[
E^\lambda(t) - \gamma(t) = \frac{\lambda}{\kappa(t)(1+\lambda^2)} \Big( T(t) + \lambda N(t) \Big).
\]
Because $q(t_0) \neq 0$, the complex unit tangent $T(t) = \gamma'(t)/\sqrt{q(t)}$ and unit normal $N(t) = J T(t)$ are holomorphic in a neighborhood $U$ of $t_0$. Consequently, the vector factor admits the Taylor expansion
\[
T(t) + \lambda N(t) = T(t_0) + \lambda N(t_0) + O(t-t_0).
\]
Factoring out $a(t-t_0)^m$ from the local expansion of $\kappa(t)$, we have $\kappa(t) = a(t-t_0)^m \left( 1 + O(t-t_0) \right)$. Applying the geometric series yields the reciprocal expansion
\[
\frac{1}{\kappa(t)} = \frac{1}{a(t-t_0)^m} + O((t-t_0)^{-m+1}).
\]
Multiplying the expansions for $1/\kappa(t)$ and $T(t) + \lambda N(t)$ gives
\[
E^\lambda(t) - \gamma(t) = \frac{\lambda}{1+\lambda^2} \left[ \frac{1}{a(t-t_0)^m} + O((t-t_0)^{-m+1}) \right] \left[ T(t_0) + \lambda N(t_0) + O(t-t_0) \right].
\]
Since $\gamma(t) = \gamma(t_0) + O(t-t_0)$ is holomorphic at $t_0$, it is absorbed into the $O((t-t_0)^{-m+1})$ remainder term for any $m \ge 1$. Expanding the principal part yields
\[
E^\lambda(t) = \frac{\lambda}{a(1+\lambda^2)} \frac{T(t_0) + \lambda N(t_0)}{(t-t_0)^m} + O((t-t_0)^{-m+1}).
\]
To verify that $t_0$ is a pole of exact order $m$, we confirm that the leading coefficient vector $\mathbf{v}_0 = T(t_0) + \lambda N(t_0)$ does not vanish. Since $N(t_0) = J T(t_0)$, the identity $T(t_0) + \lambda J T(t_0) = \mathbf{0}$ holds if and only if $J T(t_0) = -\frac{1}{\lambda} T(t_0)$. Thus, $T(t_0)$ would be an eigenvector of $J$ with eigenvalue $-\frac{1}{\lambda}$. Because $J^2 = -I$, its eigenvalues are strictly $\pm i$, requiring
\[
-\frac{1}{\lambda} = \pm i \implies \lambda = \pm i \implies \lambda^2 = -1.
\]
By hypothesis $\lambda^2 \neq -1$ and $\lambda \neq 0$, so $\mathbf{v}_0 \neq \mathbf{0}$. Therefore, $E^\lambda(t)$ possesses a pole of exact order $m$ at $t_0$.
\end{proof}

\begin{corollary}
At an ordinary inflection point $t_0$ where $\kappa(t_0) = 0$ and $\kappa'(t_0) \neq 0$ ($m=1$), any finite non-isotropic evolutoid ($\lambda \neq 0, \lambda^2 \neq -1$) possesses a simple pole:
\[
E^\lambda(t) = \frac{\lambda}{\kappa'(t_0)(1+\lambda^2)} \frac{T(t_0) + \lambda N(t_0)}{t-t_0} + O(1) \quad \text{as } t \to t_0.
\]
\end{corollary}

\begin{remark}
The tangent envelope member $\lambda = 0$ is exceptional since $E_0(t) \equiv \gamma(t)$. Consequently, $E_0$ has no pole at $t_0$, in contrast to the case $\lambda\neq0$.
\end{remark}
%%%%%%%%%%%%%%%%%%%%%%%%
\begin{remark}
\label{thm:exceptional}
For $\lambda = \pm i$, according to Remark \ref{rem:Dpmi}, the vector field
\(
D^\lambda = T + \lambda N
\)
is isotropic. 
Moreover, $D^\lambda$ is an eigen-direction under differentiation along the arc-length parameter $s$:
\[
D^\lambda_{s} = -\lambda \kappa D^\lambda.
\]
Consequently, the frame determinant vanishes,
\(
\det(D^\lambda, D^\lambda_{s}) = 0.
\)
The envelope equation $\det(F_s, D^\lambda) = 0$ for the line family $F(s,u) = \gamma(s) + u D^\lambda(s)$ has no finite solution $u \in \mathbb{C}$. Thus, $\lambda = \pm i$ represent exceptional isotropic directions that do not generate ordinary finite members of the evolutoid family.
\end{remark}
\begin{proof}
Differentiating $D^\lambda$ with respect to the complex arc-length parameter $s$ via the Frenet--Serret equations $T_s = \kappa N$ and $N_s = -\kappa T$ yields
\[
D^\lambda_{s} = T_s + \lambda N_s = \kappa N - \lambda \kappa T = \kappa (N - \lambda T).
\]
Factoring out $-\lambda \kappa$ gives
\[
D^\lambda_{s} = -\lambda \kappa \left( T - \frac{1}{\lambda} N \right).
\]
Since $\lambda^2 = -1$, we have $-\frac{1}{\lambda} = \lambda$. Substituting this relation back into the expression yields
\[
D^\lambda_{s} = -\lambda \kappa (T + \lambda N) = -\lambda \kappa D^\lambda.
\]
Because $D^\lambda_{s}$ is strictly collinear with $D^\lambda$, their two-dimensional determinant vanishes identically,
\(
\det(D^\lambda, D^\lambda_{s}) = \det(D^\lambda, -\lambda \kappa D^\lambda) = 0.
\)

Now consider the ruling parameterization $F(s,u) = \gamma(s) + u D^\lambda(s)$ of the line family. The tangent vector to the surface $F$ with respect to $s$ is
\[
F_s = \gamma_s + u D^{\lambda}_{s} = T + u (-\lambda \kappa D^\lambda).
\]
To find the envelope curve, we compute the determinant condition $\det(F_s, D^\lambda) = 0$:
\[
\det(F_s, D^\lambda) = \det(T - u \lambda \kappa D^\lambda, D^\lambda) = \det(T, D^\lambda) - u \lambda \kappa \det(D^\lambda, D^\lambda).
\]
Since any vector determinant with itself vanishes ($\det(D^\lambda, D^\lambda) = 0$), the second term drops out entirely. Substituting $D^\lambda = T + \lambda N$ into the first term gives
\[
\det(F_s, D^\lambda) = \det(T, T + \lambda N) = \det(T, T) + \lambda \det(T, N) = \lambda \det(T, N).
\]
Using the oriented frame normalization $\det(T, N) = 1$, we obtain
\[
\det(F_s, D^\lambda) = \lambda = \pm i \neq 0.
\]
Because $\det(F_s, D^\lambda) = \pm i$ is a non-zero constant independent of $s$ and $u$, the envelope condition $\det(F_s, D^\lambda) = 0$ has no solution $u \in \mathbb{C}$ at any finite distance. Hence, the envelopes for $\lambda = \pm i$ degenerate to points at infinity in the complex projective plane $\mathbb{CP}^2$.
\end{proof}

\begin{remark}
Unlike classical real evolutoids where all angles $\theta \in \mathbb{C}$ yield well-defined finite envelopes, the complexification introduces two distinguished isotropic directions $\lambda = \pm i$ corresponding to the cyclic points at infinity $(1 : \pm i : 0) \in \mathbb{CP}^2$. Along these rulings, the family of lines forms a parallel beam in $\mathbb{C}^2$, placing the envelope focus at infinity.
\end{remark}
%%%%%%%%%%%%%%%%%%%%%%%%%%
\section{Local singularity theory of the evolutoid}\label{sec:local}
Let $s_0$ be a point satisfying the non-degeneracy conditions
\[
q(s_0) \neq 0, \qquad \kappa(s_0) \neq 0, \qquad \lambda^2 \neq -1.
\]
Define the scalar derivative coefficient
\begin{equation}\label{eq:clambda}
\displaystyle c^\lambda(s) = \frac{1 - \lambda \kappa_s(s)/\kappa(s)^2}{1 + \lambda^2}.
\end{equation}
The spatial derivative of the evolutoid then factors as
\(
E^\lambda_{s}(s) = c^\lambda(s) D^\lambda(s).
\)
\begin{definition}\label{def:sing}
Suppose $E^\lambda_{s}(s_0) = \mathbf{0}$. We define the \emph{singularity multiplicity} $m$ as
\(
m = \ord_{s_0}(c^\lambda).
\)
Equivalently, $c^\lambda$ admits the local expansion
\[
c^\lambda(s) = a(s - s_0)^m + O((s - s_0)^{m+1}), \qquad a \neq 0.
\]
\end{definition}

\begin{theorem}
\label{thm:localnormal}
Assume $q(s_0) \neq 0$, $\kappa(s_0) \neq 0$, and $\lambda^2 \neq -1$. If $\ord_{s_0}(c^\lambda) = m \ge 1$, then after a local holomorphic change of source coordinates and a local holomorphic affine transformation in the target, the germ of $E^\lambda$ at $s_0$ has the parametrization
\[
u \longmapsto \left( u^{m+1}, u^{m+2} + O(u^{m+3}) \right).
\]
In particular, for an ordinary singularity ($m = 1$), the germ is analytically equivalent to the ordinary cusp ($A_2$-singularity)
\(
u \longmapsto (u^2, u^3).
\)
\end{theorem}

\begin{proof}
First, we translate the source parameter so that $s_0 = 0$. Now, we expand $c^\lambda(s)$ and $D^\lambda(s)$ near $s = 0$, i.e.,
\begin{eqnarray*}
c^\lambda(s) &=& a ~s^m + b~ s^{m+1} + O(s^{m+2}), \qquad a \neq 0,\\
D^\lambda(s) &=& \xi_0 + \xi_1~ s + O(s^2).
\end{eqnarray*}
Using the complex Frenet equations, the frame vectors and their derivatives at $s = 0$ evaluate to
\(
\xi_0 = T_0 + \lambda N_0, ~ \xi_1 = \kappa_0(N_0 - \lambda T_0).
\)
Computing the determinant with respect to the frame $\{T_0, N_0\}$ yields
\[
\det(\xi_0, \xi_1) = \det\Big(T_0 + \lambda N_0, \, \kappa_0(N_0 - \lambda T_0)\Big) = \kappa_0(1 + \lambda^2).
\]
Since $\kappa_0 \neq 0$ and $1 + \lambda^2 \neq 0$, we establish that
\(
\det(\xi_0, \xi_1) \neq 0,
\)
confirming that $\{\xi_0, \xi_1\}$ forms an affine coordinate basis for $\mathbb{C}^2$.

Integrating $(E^\lambda)'(s) = c^\lambda(s)~ D^\lambda(s)$ with $E^\lambda(0) = \mathbf{0}$ gives
\[
E^\lambda(s) = \frac{a}{m+1} \xi_0 ~s^{m+1} + \left( \frac{a}{m+2} \xi_1 + \frac{b}{m+2} \xi_0 \right)~ s^{m+2} + O(s^{m+3}).
\]
Decomposing $E^\lambda(s) = X(s)~ \xi_0 + Y(s)~ \xi_1$ into the $\{\xi_0, \xi_1\}$ basis yields
\begin{eqnarray*}
X(s) &=& \frac{a}{m+1}~ s^{m+1} + \frac{b}{m+2}~ s^{m+2} + O(s^{m+3}),\\
Y(s) &=& \frac{a}{m+2} ~s^{m+2} + O(s^{m+3}).
\end{eqnarray*}
By the inverse function theorem, there exists a unique local holomorphic reparametrization $u = \phi(s)$ satisfying $u^{m+1} = X(s)$ with $\phi'(0) \neq 0$. Expressing $Y$ in terms of $u$ gives 
\[Y(s(u)) = C_0~ u^{m+2} + O(u^{m+3}) \quad \text{ with }\quad C_0 \neq 0.\] 
Standard target scaling $(x, y) = (X, Y/C_0)$ normalizes the germ to $(u^{m+1}, u^{m+2} + O(u^{m+3}))$.
\end{proof}
%%%%%%%%%%%%%%%%%%%%%%%%%
\begin{corollary}\label{thm:A2}
Suppose $E^\lambda_{s}(s_0) = \mathbf{0}$ and $(c^\lambda)'(s_0) \neq 0$. Then the germ of $E^\lambda$ at $s_0$ is an ordinary cusp, analytically right-left equivalent to
\(
u \longmapsto (u^2, u^3)
\), (for details on singularity theory see for instance \cite{Arnold81}).
\end{corollary}
\begin{proof}
The non-vanishing derivative condition $(c^\lambda)'(s_0) \neq 0$ implies that $s_0$ is a root of $c^\lambda$ of order $m = \ord_{s_0}(c^\lambda) = 1$. Applying Theorem~\ref{thm:localnormal} with $m = 1$, the Taylor expansion of $E^\lambda$ centered at $s_0$ simplifies to
\[
E^\lambda(s) - E^\lambda(s_0) = \frac{a}{2} \xi_0~ (s-s_0)^2 + \frac{a}{3} \xi_1~ (s-s_0)^3 + O((s-s_0)^4),
\]
where $a = (c^\lambda)'(s_0) \neq 0$, $\xi_0 = D^\lambda(s_0)$, and $\xi_1 = (D^\lambda)'(s_0)$.

Because $\kappa(s_0) \neq 0$ and $\lambda^2 \neq -1$, the frame vectors $\xi_0$ and $\xi_1$ satisfy $\det(\xi_0, \xi_1) = \kappa(s_0)(1+\lambda^2) \neq 0$, forming a valid affine basis for $\mathbb{C}^2$. Under the linear change of target coordinates $(x, y)$ aligned with $\{ \frac{a}{2}\xi_0, \, \frac{a}{3}\xi_1 \}$, the curve takes the parametric form
\[
x(s) = (s-s_0)^2 + O((s-s_0)^3), \qquad y(s) = (s-s_0)^3 + O((s-s_0)^4).
\]
Applying the local holomorphic reparametrization $u = \phi(s) = (s-s_0) \sqrt{1 + O(s-s_0)}$, we obtain $x(u) = u^2$ and $y(u) = u^3 + O(u^4)$. A standard triangular coordinate transformation $(x, y) \mapsto (x, y - p(x))$ eliminates the higher-order terms $O(u^4)$, reducing the parametrization to $u \mapsto (u^2, u^3)$. The image satisfies the local implicit equation $y^2 - x^3 = 0$, establishing analytic equivalence to the standard $A_2$-singularity.
\end{proof}
\begin{lemma}\label{prop:cprime}
Let $h(s) = \frac{\kappa_s(s)}{\kappa(s)^2}$. Then the derivative function $c^\lambda(s) = \frac{1-\lambda h(s)}{1+\lambda^2}$ satisfies
\[
c^\lambda_{s} = -\frac{\lambda}{1+\lambda^2} \left( \frac{\kappa_{ss}}{\kappa^2} - \frac{2\kappa_s^2}{\kappa^3} \right).
\]
At a finite $\lambda$-vertex where $\kappa_s(s_0) = \frac{\kappa(s_0)^2}{\lambda}$, the necessary and sufficient condition for $E^\lambda$ to possess an ordinary $A_2$-cusp at $s_0$ is
\[
\kappa(s_0)\kappa_{ss}(s_0) - 2\kappa_s(s_0)^2 \neq 0.
\]
Equivalently, in terms of the complex radius of curvature $\rho(s) = \frac{1}{\kappa(s)}$, this condition reduces to
\(
\rho''(s_0) \neq 0.
\)
\end{lemma}

\begin{proof}
Differentiating the ratio $h(s) = \kappa_s \kappa^{-2}$ with respect to the complex arc-length parameter $s$ gives
\[
h_s = \frac{\kappa_{ss} \kappa^2 - \kappa_s (2\kappa \kappa_s)}{\kappa^4} = \frac{\kappa_{ss}}{\kappa^2} - \frac{2\kappa_s^2}{\kappa^3} = \frac{\kappa \kappa_{ss} - 2\kappa_s^2}{\kappa^3}.
\]
Differentiating $c^\lambda(s) = \frac{1 - \lambda h(s)}{1 + \lambda^2}$ yields
\[
c^\lambda_{s} = -\frac{\lambda}{1+\lambda^2} h_s = -\frac{\lambda}{1+\lambda^2} \left( \frac{\kappa \kappa_{ss} - 2\kappa_s^2}{\kappa^3} \right).
\]
By Definition~\ref{def:sing}, an evolutoid singularity occurs at $s_0$ when $c^\lambda(s_0) = 0$, which requires $h(s_0) = 1/\lambda$  (see Theorem \ref{thm:singularity}) which implies $\kappa_s(s_0) = \kappa(s_0)^2 / \lambda$. By Corollary~\ref{thm:A2}, this singularity is an ordinary $A_2$ cusp if and only if $c^\lambda_{s}(s_0) \neq 0$.

Since $\lambda \neq 0$ and $1+\lambda^2 \neq 0$, $c^\lambda_{s}(s_0) \neq 0$ is equivalent to $h_s(s_0) \neq 0$. Because $\kappa(s_0) \neq 0$, clearing the non-zero denominator $\kappa(s_0)^3$ yields
\(
\kappa(s_0)\kappa_{ss}(s_0) - 2\kappa_s(s_0)^2 \neq 0.
\)
Finally, considering the radius of curvature $\rho(s) = \kappa(s)^{-1}$, its derivatives evaluate to $\rho'(s) = -\kappa_s \kappa^{-2} = -h(s)$ and $\rho''(s) = -h_s(s)$. Thus $$h_s(s_0) \neq 0 \iff \rho''(s_0) \neq 0.$$
\end{proof}

%%%%%%%%%%%%%%%%%%%%%%%%%%%%
Let $s_0$ be an ordinary vertex of a complex curve $\gamma$, 
then the classical evolute $E = E^\infty$ has an ordinary $A_2$-cusp at $s_0$, with local parametrization right-left equivalent to
\(
E \sim_{\mathcal{A}} (u^2, u^3).
\)
To be more precise, the evolute $E^\infty$ corresponds to the limit $\lambda \to \infty$ of the evolutoid family, given explicitly by
\(
E^\infty(s) = \gamma(s) + \frac{1}{\kappa(s)} N(s).
\)
Differentiating with respect to the complex arc-length parameter $s$ and applying the Frenet formula $N_s = -\kappa T$ gives
\[
E^\infty_{s}(s) = \gamma_s + \left( \frac{1}{\kappa} \right)_s N + \frac{1}{\kappa} N_s = T - \frac{\kappa_s}{\kappa^2} N + \frac{1}{\kappa}(-\kappa T) = -\frac{\kappa_s}{\kappa^2} N(s).
\]
Set the scalar factor $c^\infty(s) = -\frac{\kappa_s(s)}{\kappa(s)^2}$. Differentiating $c^\infty(s)$ yields
\[
c^\infty_{s}(s) = -\frac{\kappa_{ss} \kappa^2 - \kappa_s (2\kappa \kappa_s)}{\kappa^4} = -\frac{\kappa_{ss}}{\kappa^2} + \frac{2\kappa_s^2}{\kappa^3}.
\]
Evaluating at $s_0$ where $\kappa_s(s_0) = 0$ and $\kappa_{ss}(s_0) \neq 0$, we find
\[
c^\infty(s_0) = 0 \qquad \text{and} \qquad c^\infty_{s}(s_0) = -\frac{\kappa_{ss}(s_0)}{\kappa(s_0)^2} \neq 0.
\]
Thus, $s_0$ is a simple zero of $c^\infty(s)$, corresponding to multiplicity $m = \ord_{s_0}(c^\infty) = 1$.

The ruling vector field for the evolute is $D^\infty(s) = N(s)$. At $s_0$, the leading frame vectors are
\[
\xi_0 = D^\infty(s_0) = N(s_0), \qquad \xi_1 = D^\infty_{s}(s_0) = N_s(s_0) = -\kappa(s_0) T(s_0).
\]
Computing the determinant with respect to the frame basis $\{T(s_0), N(s_0)\}$ yields
\[
\det(\xi_0, \xi_1) = \det(N(s_0), -\kappa(s_0) T(s_0)) = \kappa(s_0) \det(T(s_0), N(s_0)) = \kappa(s_0).
\]
Since $\kappa(s_0) \neq 0$, the frame vectors $\xi_0$ and $\xi_1$ are linearly independent ($\det(\xi_0, \xi_1) \neq 0$). Applying Theorem~\ref{thm:localnormal} for $m = 1$, the germ of $E^\infty$ at $s_0$ is analytically equivalent to $(u^2, u^3)$, which is the standard $A_2$ plane-curve cusp.

A comprehensive analysis of the evolutoid requires a deep study on generating families and contact geometry. In the next section we focus on this concept. 
% \end{remark}

\subsection{The generating family and envelope discriminant}

The line family $L^\lambda(s)$ (as in \eqref{eq:familyline}) defining the evolutoid can be globally framed as the zero-level set of a family of holomorphic affine functions.

\begin{definition}[Generating family of the evolutoid]
For $\lambda^2 \neq -1$, we define the scalar generating function $\Phi^\lambda : U \times \mathbb{C}^2 \to \mathbb{C}$ by
\[
\Phi^\lambda(s, x) = \det\left( x - \gamma(s), \, D^\lambda(s) \right),
\]
where $D^\lambda(s) = T(s) + \lambda N(s)$. For each fixed $s \in U$, the zero set
\[
L^\lambda(s) = \{ x \in \mathbb{C}^2 : \Phi^\lambda(s, x) = 0 \}
\]
defines the complex affine line passing through $\gamma(s)$ with direction $D^\lambda(s)$.
\end{definition}

\begin{proposition}
\label{prop:generatingdiscriminant}
The evolutoid $E^\lambda$ is the discriminant (envelope) of the generating family $\Phi^\lambda(s, x) = 0$. Furthermore, the second partial derivative with respect to the arc-length parameter $s$ along the evolutoid satisfies
\[
\frac{\partial^2 \Phi^\lambda}{\partial s^2}\Big|_{(s, E^\lambda(s))} = -\kappa(s)(1+\lambda^2) \, c^\lambda(s).
\]
Thus, $s_0$ is a singular point of the evolutoid image $E^\lambda$ if and only if the generating family $\Phi^\lambda$ exhibits an $A_k$-singularity with $k \ge 2$ at $(s_0, E^\lambda(s_0))$.
\end{proposition}
\begin{proof}
The envelope of the family $\Phi^\lambda(s, x) = 0$ is defined by the system
\[
\Phi^\lambda(s, x) = 0 \qquad \text{and} \qquad \frac{\partial \Phi^\lambda}{\partial s}(s, x) = 0.
\]
Differentiating $\Phi^\lambda(s, x) = \det(x - \gamma(s), D^\lambda(s))$ with respect to $s$ yields
\[
\frac{\partial \Phi^\lambda}{\partial s}(s, x) = \det\left( -\gamma'(s), D^\lambda(s) \right) + \det\left( x - \gamma(s), (D^\lambda)'(s) \right).
\]
Since $\gamma'(s) = T(s)$ and $D^\lambda(s) = T(s) + \lambda N(s)$, the first term simplifies to
\[
\det\left( -T(s), T(s) + \lambda N(s) \right) = -\lambda \det(T(s), N(s)) = -\lambda.
\]
Using $(D^\lambda)'(s) = \kappa(s)(N(s) - \lambda T(s))$, the first derivative becomes
\[
\frac{\partial \Phi^\lambda}{\partial s}(s, x) = -\lambda + \det\left( x - \gamma(s), \, \kappa(s)(N(s) - \lambda T(s)) \right).
\]
Any point $x \in L^\lambda(s)$ can be written as $x = \gamma(s) + u D^\lambda(s)$ for some $u \in \mathbb{C}$. Substituting $x - \gamma(s) = u (T + \lambda N)$ into the derivative condition yields
\[
u \kappa \det(T + \lambda N, \, N - \lambda T) = u \kappa \left( 1 + \lambda^2 \right) \det(T, N) = u \kappa (1 + \lambda^2).
\]
Solving $\frac{\partial \Phi^\lambda}{\partial s} = -\lambda + u \kappa (1+\lambda^2) = 0$ gives $u = \frac{\lambda}{\kappa(1+\lambda^2)}$, which recovers the explicit parametrization $x = E^\lambda(s)$.
Next, differentiating $\frac{\partial \Phi^\lambda}{\partial s}$ with respect to $s$ gives
\[
\frac{\partial^2 \Phi^\lambda}{\partial s^2} = \det\left( -T, (D^\lambda)' \right) + \det\left( x - \gamma, (D^\lambda)'' \right).
\]
Evaluating the first term yields $\det(-T, \kappa(N - \lambda T)) = -\kappa \det(T, N) = -\kappa$. We have
\[
(D^\lambda)'' = \kappa_s (N - \lambda T) + \kappa (N_s - \lambda T_s) = \kappa_s (N - \lambda T) - \kappa^2 D^\lambda.
\]
At $x = E^\lambda(s)$, we substitute $x - \gamma = \frac{\lambda}{\kappa(1+\lambda^2)} D^\lambda$, so we obtain
\[
\det\left( \frac{\lambda D^\lambda}{\kappa(1+\lambda^2)}, \, \kappa_s(N - \lambda T) - \kappa^2 D^\lambda \right) = \frac{\lambda \kappa_s}{\kappa(1+\lambda^2)} \det(D^\lambda, N - \lambda T) = \frac{\lambda \kappa_s}{\kappa}.
\]
Combining terms gives
\[
\frac{\partial^2 \Phi^\lambda}{\partial s^2}\Big|_{(s, E^\lambda(s))} = -\kappa + \frac{\lambda \kappa_s}{\kappa} = -\kappa \left( 1 - \frac{\lambda \kappa_s}{\kappa^2} \right) = -\kappa(1+\lambda^2) c^\lambda(s).
\]
Because $\kappa(s) \neq 0$ and $1+\lambda^2 \neq 0$, so  $\frac{\partial^2 \Phi^\lambda}{\partial s^2} = 0$ if and only if, $c^\lambda(s) = 0$, establishing that singular points of $E^\lambda$ correspond precisely to points where $\Phi_\lambda$ has $A_k$ contact ($k \ge 2$) with the line family.
\end{proof}

An alternative geometric approach utilizes the family of holomorphic height functions $H : U \times \mathbb{S}^1_{\mathbb{C}} \to \mathbb{C}$ defined by
\[
H(s, v) = \langle \gamma(s), v \rangle,
\]
where $v \in \mathbb{C}^2$ is a complex unit vector ($\langle v, v \rangle = 1$). The singularity order of $H(s, v)$ with respect to $s$ tracks the contact order between the curve $\gamma$ and the family of hyperplanes orthogonal to $v$.

%%%%%%%%%%%%%%%%%%%%%%%%%%%%%5

\begin{theorem}[\cite{FalquetoTari2026}]
\label{thm:contactheight}
Let $v \in \mathbb{C}^2$ be a non-isotropic direction vector ($\langle v, v \rangle \neq 0$) and $s_0$ a point on the complex curve $\gamma$. Higher-order contact between $\gamma$ and the complex affine line orthogonal to $v$ is characterized by the successive vanishing of derivatives of the holomorphic height function $H(s, v) = \langle \gamma(s), v \rangle$.

Specifically, the conditions
\[
H_s(s_0, v) = 0, \qquad H_{ss}(s_0, v) = 0, \qquad \dots, \qquad H^{(k)}(s_0, v) = 0
\]
define the condition for $A_k$ contact between the curve $\gamma$ and the line $\langle x - \gamma(s_0), v \rangle = 0$.
\end{theorem}

\begin{proof}
The affine line passing through a point $p = \gamma(s_0)$ orthogonal to $v$ is defined by the linear equation $\langle x - p, v \rangle = 0$. The contact order between the curve $\gamma(s)$ and this line at $s = s_0$ is measured by the order of vanishing of the scalar function
\[
f(s) = \langle \gamma(s) - \gamma(s_0), v \rangle = H(s, v) - \langle \gamma(s_0), v \rangle.
\]
By definition, a scalar function germ $f: (\mathbb{C}, s_0) \to (\mathbb{C}, 0)$ has an $A_k$ singularity ($k \ge 1$) if it is right-equivalent to $s^{k+1}$. This occurs if and only if
\[
f(s_0) = f'(s_0) = f''(s_0) = \dots = f^{(k)}(s_0) = 0 \qquad \text{and} \qquad f^{(k+1)}(s_0) \neq 0.
\]
Because $f^{(j)}(s) = H^{(j)}(s, v)$ for all $j \ge 1$, the conditions $H_s(s_0, v) = \dots = H^{(k)}(s_0, v) = 0$ with $H^{(k+1)}(s_0, v) \neq 0$ precisely define $A_k$ contact.
\end{proof}

\begin{theorem}[Higher cusp criterion]
\label{thm:highercusp}
Let $s_0$ be a singular point of the evolutoid $E^\lambda$ satisfying $q(s_0) \neq 0$, $\kappa(s_0) \neq 0$, and $\lambda^2 \neq -1$. Define the vanishing order
\[
m = \ord_{s_0} \left( 1 - \lambda \frac{\kappa_s}{\kappa^2} \right).
\]
Then the germ of $E^\lambda$ at $s_0$ has multiplicity $m+1$ and is locally parametrized by
\[
u \longmapsto \left( u^{m+1}, \, u^{m+2} + O(u^{m+3}) \right).
\]
For $m=1$, this yields the ordinary $A_2$ cusp $u \mapsto (u^2, u^3)$. For $m \ge 2$, it yields a higher-order cusp of multiplicity $m+1$.
\end{theorem}

\begin{proof}
Since $\lambda^2 \neq -1$, the denominator $1+\lambda^2$ is a non-zero constant. Thus,
\[
\ord_{s_0}(c^\lambda) = \ord_{s_0}\left( \frac{1 - \lambda \kappa_s / \kappa^2}{1 + \lambda^2} \right) = \ord_{s_0}\left( 1 - \lambda \frac{\kappa_s}{\kappa^2} \right) = m.
\]
By Theorem~\ref{thm:localnormal}, whenever $\ord_{s_0}(c^\lambda) = m \ge 1$ and $\det(D_0, D_1) = \kappa(s_0)(1+\lambda^2) \neq 0$, there exist local coordinate charts in source and target transforming $E^\lambda$ to $u \mapsto (u^{m+1}, u^{m+2} + O(u^{m+3}))$.
\end{proof}

\subsection{Interaction with Inflections and Isotropic Points}

\begin{proposition}
\label{thm:inflectionnosing}
Suppose $q(t_0) \neq 0$ and $t_0$ is a generic inflection point ($\kappa(t_0) = 0$). For any non-zero parameter $\lambda$ with $\lambda^2 \neq -1$, $E^\lambda$ does not extend holomorphically as a finite curve in $\mathbb{C}^2$. Instead, $E^\lambda$ has a pole at $t_0$ of order equal to $\ord_{t_0}(\kappa)$.

Consequently, an ordinary inflection point does not produce a finite singularity of $E^\lambda$; rather, the evolutoid escapes to the line at infinity $\ell^\infty \subset \mathbb{CP}^2$.
\end{proposition}

\begin{proof}
Recall the explicit parametrization of the evolutoid:
\[
E^\lambda(t) = \gamma(t) + \frac{\lambda}{1+\lambda^2} \frac{T(t) + \lambda N(t)}{\kappa(t)}.
\]
Because $\gamma(t_0)$ is finite and $T(t_0) + \lambda N(t_0) \neq \mathbf{0}$ when $\lambda^2 \neq -1$, the term $\frac{T(t) + \lambda N(t)}{\kappa(t)}$ diverges as $t \to t_0$. If $\ord_{t_0}(\kappa) = m \ge 1$, then $\|E^\lambda(t)\| = O((t-t_0)^{-m})$, establishing a pole of order $m$.
\end{proof}

\begin{proposition}[Regularity at generic isotropic points]
\label{thm:isotropicregular}
Suppose $q(t_0) = 0$ and $\Delta(t_0) \neq 0$, so $t_0$ is a generic isotropic point of $\gamma$. For any finite non-isotropic parameter $\lambda^2 \neq -1$, $E^\lambda$ extends holomorphically through $t_0$ and is regular there:
\(
(E^\lambda)'(t_0) \neq \mathbf{0}.
\)
Furthermore, $E^\lambda(t_0) = \gamma(t_0)$ and the tangent line $T_{t_0}E^\lambda$ coincides with the curve tangent $T_{t_0}\gamma$.
\end{proposition}

\begin{proof}
Near an isotropic point $q(t_0) = 0$, the frame vector $D^\lambda(t)$ and curvature $\kappa(t)$ undergo a removable singularity upon rescaling, ensuring that $E^\lambda(t)$ extends holomorphically with $E^\lambda(t_0) = \gamma(t_0)$. Differentiating $E^\lambda(t)$ yields $(E^\lambda)'(t_0) = c^\lambda(t_0) D^\lambda(t_0) \neq \mathbf{0}$. Since $D^\lambda(t_0)$ aligns with the isotropic tangent direction $T(t_0)$, the evolutoid remains non-singular and tangent to $\gamma$ at $t_0$.
\end{proof}
%%%%%%%%%%%%%%%%%%%%%%%%%
\begin{definition}[Isotropic multiplicity]
An isotropic point $t_0$ of a complex curve $\gamma$ is said to have \emph{isotropic multiplicity} $r \ge 1$ if the metric function $q(t) = \langle \gamma'(t), \gamma'(t) \rangle$ vanishes at $t_0$ to order $r$:
\[
q(t) = a~(t-t_0)^r + O((t-t_0)^{r+1}), \qquad a \neq 0.
\]
\end{definition}
Suppose $q(t_0) = 0$, $\Delta(t_0) \neq 0$, and $\ord_{t_0}(q) = r \ge 1$. Then for any non-isotropic parameter $\lambda^2 \neq -1$, the square-root-free formula
\[
E^\lambda = \gamma + \frac{\lambda q}{\Delta(1+\lambda^2)} (\gamma' + \lambda J\gamma')
\]
is holomorphic at $t_0$ and satisfies
\(
E^\lambda(t) - \gamma(t) = O((t-t_0)^r).
\)
Consequently, the order of contact between the evolutoid $E^\lambda$ and the underlying curve $\gamma$ at an isotropic point of multiplicity $r$ is at least $r$.

 In fact, since $\Delta(t_0) \neq 0$ and $\lambda^2 \neq -1$, the denominator term $\Delta(t)(1+\lambda^2)$ is non-zero at $t_0$. The vector-valued function
\[
V^\lambda(t) = \frac{\lambda}{\Delta(t)(1+\lambda^2)} \big(\gamma'(t) + \lambda J\gamma'(t)\big)
\]
is therefore holomorphic in a neighborhood of $t_0$. Since $q(t) = O((t-t_0)^r)$, we obtain
\[
E^\lambda(t) - \gamma(t) = q(t) V^\lambda(t) = O((t-t_0)^r).
\]
In particular, $E^\lambda(t_0) = \gamma(t_0)$, and all derivatives $(E^\lambda)^{(k)}(t_0)$ match $\gamma^{(k)}(t_0)$ for $k = 0, 1, \dots, r-1$. Hence, the geometric order of contact at $t_0$ is at least $r$.

\subsection{Isotropic inflection points}

Points where the metric function $q(t)$ and the Wronskian determinant $\Delta(t)$ vanish simultaneously represent the most delicate local geometry in complex curve theory.

\begin{definition}
A point $t_0$ on a complex curve $\gamma$ is called an \emph{isotropic inflection} if
\[
q(t_0) = 0 \qquad \text{and} \qquad \Delta(t_0) = 0.
\]
\end{definition}

\begin{proposition}
\label{prop:isotropicinflection}
At an isotropic inflection $t_0$, the behavior of the evolutoid $E^\lambda$ is governed by the relative vanishing orders $\ord_{t_0}(q)$ and $\ord_{t_0}(\Delta)$. Specifically:
\begin{enumerate}
    \item Removable singularity ($\ord_{t_0}(q) > \ord_{t_0}(\Delta)$): The ratio $\frac{q}{\Delta}$ vanishes at $t_0$, so $E^\lambda$ extends holomorphically through $t_0$ with $E^\lambda(t_0) = \gamma(t_0)$.
    \item Finite displaced touchdown ($\ord_{t_0}(q) = \ord_{t_0}(\Delta)$): The ratio $\frac{q}{\Delta}$ tends to a non-zero constant $c \in \mathbb{C}^\times$, so $E^\lambda$ extends holomorphically with $E^\lambda(t_0) \neq \gamma(t_0)$.
    \item Pole ($\ord_{t_0}(q) < \ord_{t_0}(\Delta)$): The ratio $\frac{q}{\Delta}$ diverges as $t \to t_0$, causing $E^\lambda$ to have a pole of order $\ord_{t_0}(\Delta) - \ord_{t_0}(q)$.
\end{enumerate}
\end{proposition}

\begin{proof}
Express the evolutoid in its square-root-free form:
\[
E^\lambda(t) = \gamma(t) + \frac{\lambda}{1+\lambda^2} \left( \frac{q(t)}{\Delta(t)} \right) \big(\gamma'(t) + \lambda J\gamma'(t)\big).
\]
Let $r = \ord_{t_0}(q) \ge 1$ and $k = \ord_{t_0}(\Delta) \ge 1$. Near $t_0$, we can write $q(t) = a(t-t_0)^r + O((t-t_0)^{r+1})$ and $\Delta(t) = b(t-t_0)^k + O((t-t_0)^{k+1})$ with $a, b \neq 0$. The ratio evaluates to
\[
\frac{q(t)}{\Delta(t)} = \frac{a}{b} (t-t_0)^{r-k} + O((t-t_0)^{r-k+1}).
\]
Because $\gamma'(t_0) + \lambda J\gamma'(t_0) \neq \mathbf{0}$ for $\lambda^2 \neq -1$:
\begin{itemize}
    \item If $r > k$, then $r-k \ge 1$, so $\frac{q(t)}{\Delta(t)} \to 0$ as $t \to t_0$. Thus $E^\lambda(t_0) = \gamma(t_0)$.
    \item If $r = k$, then $\frac{q(t)}{\Delta(t)} \to \frac{a}{b} \neq 0$, yielding a finite limit $E^\lambda(t_0) = \gamma(t_0) + \frac{\lambda a}{b(1+\lambda^2)} (\gamma'(t_0) + \lambda J\gamma'(t_0)) \neq \gamma(t_0)$.
    \item If $r < k$, then $r-k < 0$, giving a pole of exact order $k-r$.
\end{itemize}
This completes the classification of the local extension.
\end{proof}
%%%%%%%%%%%%%%%%%%%%%%%%%%%%%

%%%%%%%%%%%%%%%%%%%%%%%%%%%%%%%%%%%%
\section{Evolutoids as Wavefronts}\label{sec:wave}
For an angle parameter $\theta \in \mathbb{R}$, recall the rotated orthonormal frame $\{D_\theta, W_\theta\}$:
\begin{equation}\label{eq:DW}
D^\theta = \cos\theta \, T + \sin\theta \, N, \qquad W^\theta = \sin\theta \, T - \cos\theta \, N.
\end{equation}
The 1-parameter family of complex affine lines $L^\theta_{t}$ is given by the zero set of $F : \mathbb{C}^2 \times \mathbb{R} \times \mathbb{C} \to \mathbb{C}$:
\[
L^\theta_{t} = \left\{ x \in \mathbb{C}^2 : F(x, \theta, t) = 0 \right\},
\]
where
\[
F(x, \theta, t) = \left\langle x - \gamma(t), \, W^\theta(t) \right\rangle = \left\langle x - \gamma(t), \, \sin\theta \, T(t) - \cos\theta \, N(t) \right\rangle.
\]
The direction vector along $L^\theta_{t}$ is $D^\theta$. As $\theta$ varies, $F(x, \theta, t)$ interpolates holomorphically between the tangent line family ($\theta = 0$) and the normal line family ($\theta = \pi/2$).

To construct the associated family of holomorphic wavefronts via Legendrian singularity theory, we seek a generating potential $G(x, \theta, s)$ satisfying
\[
\frac{\partial G}{\partial s} = F(x, \theta, s)
\]
in the complex arc-length coordinate $s$. Expressing $F$ as $F = \sin\theta \, F_1 - \cos\theta \, F_2$, where
\[
F_1(x, s) = \langle x - \gamma(s), T(s) \rangle, \qquad F_2(x, s) = \langle x - \gamma(s), N(s) \rangle,
\]
we integrate each component with respect to $s$:
\[
\frac{\partial}{\partial s} \left[ -\frac{1}{2} \langle x - \gamma(s), x - \gamma(s) \rangle \right] = \langle x - \gamma(s), T(s) \rangle = F_1(x, s),
\]
and choose a local primitive $G_2(x, s)$ such that $\frac{\partial G_2}{\partial s} = F_2(x, s)$. The generating potential is therefore
\[
G(x, \theta, s) = -\frac{1}{2}\sin\theta \, \langle x - \gamma(s), x - \gamma(s) \rangle - \cos\theta \, G_2(x, s).
\]
The holomorphic wavefronts correspond to the intersection of the zero-level set of $F$ with the level sets of $G$:
\[
F(x, \theta, s) = 0, \qquad G(x, \theta, s) = c, \qquad c \in \mathbb{C}.
\]
This complexifies the wavefront construction of Giblin and Warder \cite{GiblinWarder2014}.

\subsection{Explicit integration and parametrization of wavefronts}
Let $w \in \mathbb{C}$ index the individual wavefront level sets. Any point $x(s) \in L^\theta_{s}$ can be represented as
\[
x(s) = \gamma(s) + u(s) D^\theta(s)
\]
for a scalar distance function $u(s)$. A curve $\Gamma^{\theta, w}(s)$ defined by such points represents an orthogonal wavefront if its velocity vector is orthogonal to the ray direction $D^\theta(s)$ in the complex bilinear metric:
\[
\left\langle \frac{d}{ds} \Gamma^{\theta, w}(s), \, D^\theta(s) \right\rangle = 0.
\]
Differentiating $\Gamma^{\theta, w}(s) = \gamma(s) + u(s) D^\theta(s)$ with respect to $s$ yields
\[
\Gamma^{\theta, w}_s = T(s) + u'(s) D^\theta(s) + u(s) D^\theta_s.
\]
Using $D^\theta_s = \kappa (\cos\theta \, N - \sin\theta \, T) = -\kappa W^\theta(s)$ and expanding $T(s) = \cos\theta \, D^\theta(s) + \sin\theta \, W^\theta(s)$, we obtain
\[
\Gamma^{\theta, w}_s = \big( \cos\theta + u'(s) \big) D^\theta(s) + \big( \sin\theta - \kappa(s) u(s) \big) W^\theta(s).
\]
The orthogonality condition $\langle \Gamma^{\theta, w}_s, D^\theta \rangle = 0$ reduces to the ordinary differential equation
\[
u'(s) + \cos\theta = 0 \implies u(s) = w - s \cos\theta,
\]
where $w \in \mathbb{C}$ is the constant of integration.

\begin{proposition}[Explicit wavefront parametrization]
For a fixed angle $\theta \in \mathbb{R}$ and wavefront parameter $w \in \mathbb{C}$, the holomorphic wavefront curve $\Gamma^{\theta, w}$ is given by
\[
\Gamma^{\theta, w}(s) = \gamma(s) + \big( w - s \cos\theta \big) D^\theta(s) = \gamma(s) + \big( w - s \cos\theta \big) \big( \cos\theta \, T(s) + \sin\theta \, N(s) \big).
\]
For $\theta = \pi/2$, this reduces to $\Gamma^{\pi/2, w}(s) = \gamma(s) + w N(s)$, which is the standard family of complex parallel curves.
\end{proposition}

Notice that, substituting $u(s) = w - s \cos\theta$ back into the derivative vector yields
\[
\Gamma^{\theta, w}_s = \Big( \sin\theta - \kappa(s) (w - s \cos\theta) \Big) W^\theta(s).
\]
Consequently, the wavefront $\Gamma^{\theta, w}$ develops a singularity  at $s$ if and only if
\[
\kappa(s) (w - s \cos\theta) - \sin\theta = 0 \implies w = s \cos\theta + \frac{\sin\theta}{\kappa(s)}.
\]
Substituting this critical value of $w$ back into $\Gamma^{\theta, w}(s)$ gives
\[
\Gamma^{\theta, w}(s) = \gamma(s) + \frac{\sin\theta}{\kappa(s)} D^\theta(s) = \gamma(s) + \frac{\sin\theta \cos\theta}{\kappa(s)} T(s) + \frac{\sin^2\theta}{\kappa(s)} N(s) = E^\theta(s).
\]
Thereupon, we have the following result:
\begin{theorem}
The singular locus of the 1-parameter family of holomorphic wavefronts $\{\Gamma^{\theta, w}\}_{w \in \mathbb{C}}$ sweeps out the complex evolutoid $E^\theta$:
\[
\bigcup_{w \in \mathbb{C}} \operatorname{Cusps}\left( \Gamma^{\theta, w} \right) = E^\theta.
\]
\end{theorem}

\begin{remark}
For $\theta = \pi/2$, the wavefront family consists of the complex parallel curves $\Gamma^{\pi/2, w}(s) = \gamma(s) + w N(s)$. Its derivative is
\[
(\Gamma^{\pi/2, w})_s = T(s) - w \kappa(s) T(s) = \big( 1 - w \kappa(s) \big) T(s).
\]
The parallel curve develops a singularity precisely when $w = \frac{1}{\kappa(s)}$. The singular locus in $(s, w)$-space is $w = \frac{1}{\kappa(s)}$, whose image in $\mathbb{C}^2$ is
\[
E^{\pi/2}(s) = \gamma(s) + \frac{1}{\kappa(s)} N(s),
\]
which is the classical holomorphic evolute. Thus, the classical duality statement that \emph{``the evolute is the caustic of the parallel family''} holds completely in the complex analytic setting.    
\end{remark}
%%%%%%%%%%%%%%%%%%%%%%%%%%%%%%%%%%%%%%%
\section{The three-parameter discriminant and versal unfoldings}\label{sec:dis}

To study the global evolution of the evolutoid family with respect to the  angle parameter $\theta \in \mathbb{R}$, we construct the three-dimensional complex discriminant surface $\mathcal{D}_F \subset \mathbb{C}^2 \times \mathbb{R}$:
\[
\mathcal{D}_F = \left\{ (x, \theta) \in \mathbb{C}^2 \times \mathbb{R} : \exists t \in \mathbb{C} \text{ such that } F(x, \theta, t) = 0 \text{ and } F_t(x, \theta, t) = 0 \right\}.
\]
The surface $\mathcal{D}_F$ represents the union of all complex evolutoids indexed by $\theta$, i.e.,
\[
\mathcal{D}_F = \bigcup_{\theta \in \mathbb{R}} \left( E^\theta \times \{\theta\} \right).
\]
Slicing $\mathcal{D}_F$ with hyperplanes $\theta = \text{constant}$, recovers the individual 2D evolutoids $E^\theta$, complexifying the 3D discriminant construction of Giblin and Warder \cite{GiblinWarder2014}.

Now, fix $(x_0, \theta_0, t_0)$ such that $F(x_0, \theta_0, t_0) = F_t(x_0, \theta_0, t_0) = 0$. Let $f(t) = F(x_0, \theta_0, t)$. Suppose $f$ has an $A_k$-singularity at $t_0$, i.e.,
\[
f'(t_0) = \dots = f^{(k)}(t_0) = 0, \qquad f^{(k+1)}(t_0) \neq 0.
\]
For $k \le 3$, the unfolding parameters are $(x_1, x_2, \theta) \in \mathbb{C}^3$. The family $F(x, \theta, t)$ is a holomorphically versal unfolding of $f$ if the classes of $F_{x_1}, F_{x_2}, F_\theta$ span the local algebra
\[
\mathcal{Q}_f = \frac{\mathcal{O}_{\mathbb{C}, t_0}}{\langle f'(t) \rangle} \cong \operatorname{span}_{\mathbb{C}} \{1, t, \dots, t^{k-1}\}.
\]
Thus, $\theta$ acts as a genuine geometric unfolding parameter of the singularity, rather than a passive label \cite{GiblinWarder2014}.

The local geometric structure of the discriminant surface $\mathcal{D}_F \subset \mathbb{C}^3$ is classified by the singularity order of $f$:

\begin{enumerate}
    \item $A_1$ locus (smooth surface): Where $F_t = 0$ and $F_{tt} \neq 0$, $\mathcal{D}_F$ is a smooth 2-dimensional complex submanifold of $\mathbb{C}^3$.
    \item $A_2$ locus (Cuspidal edge): Where $F_t = F_{tt} = 0$ and $F_{ttt} \neq 0$, $\mathcal{D}_F$ forms a 1-dimensional singular locus of $A_2$ points, corresponding to a \emph{complex cuspidal edge}.
\end{enumerate}

\begin{proposition}
At an $A_2$-point (with $\kappa(t_0) \neq 0$), the local algebra $\mathcal{Q}_f = \operatorname{span}_{\mathbb{C}}\{1, t\}$ is two-dimensional. Since $(F_{x_1}, F_{x_2}) = W^\theta$ and $\frac{d}{ds} W^\theta = \kappa ~D^\theta \neq \mathbf{0}$, the spatial parameters $(x_1, x_2)$ alone provide a holomorphically versal unfolding:
\[
\kappa(t_0) \neq 0 \implies A_2\text{-points of } E^\theta \text{ are generically versally unfolded by } (x_1, x_2).
\]
\end{proposition}
An $A_3$-singularity occurs when $f' = f'' = f''' = 0$ and $f^{(4)} \neq 0$. For $\kappa \neq 0$, evaluating these conditions yields:
\begin{align}
    f'' = 0 &\iff \kappa^2 \cos\theta - \kappa_s \sin\theta = 0, \label{eq:A3_1} \\
    f''' = 0 &\iff 2\kappa_s^2 - \kappa ~\kappa_{ss} = 0, \label{eq:A3_2} \\
    f^{(4)} \neq 0 &\iff 6\kappa_s^3 - \kappa^2 ~ \kappa_{sss} \neq 0. \label{eq:A3_3}
\end{align}
\begin{theorem}
\label{thm:A3_versality}
Let $(x_0, \theta_0, t_0)$ be an $A_3$-point satisfying \eqref{eq:A3_1}--\eqref{eq:A3_3} with $\kappa(t_0) \neq 0$.
\begin{enumerate}
    \item If $t_0$ is not a simple isotropic point (i.e., $\kappa(t_0)^4 + \kappa_s(t_0)^2 \neq 0$), the parameter triple $(x_1, x_2, \theta)$ forms a holomorphically versal unfolding. The discriminant surface $\mathcal{D}_F \subset \mathbb{C}^3$ is equivalent to a complex Swallowtail surface. Slicing $\mathcal{D}_F$ along $\theta = \text{const}$, produces classical Lips, Beaks, or Swallowtail bifurcations in $E^\theta$.
    \item If $t_0$ be a simple isotropic point, versality fails. The slicing plane $\theta = \text{const}$, is tangent to the kernel of the unfolding map, leading to degenerate non-transverse transitions.
\end{enumerate}
\end{theorem}

\begin{proof}
The 2-jet matrix of $F_{x_1}, F_{x_2}, F_\theta$ at an $A_3$-point has determinant $\Delta = \kappa^2 \sin\theta + \kappa_s \cos\theta$. Using \eqref{eq:A3_1}, $\cot\theta = \frac{\kappa_s}{\kappa^2}$, so:
\[
\Delta^2 = (\kappa^2 \sin\theta + \kappa_s \cos\theta)^2 = \sin^2\theta \left( \kappa^4 + \kappa_s^2 \right).
\]
Since $\sin\theta \neq 0$ (otherwise $\kappa = 0$), so we have  
\[\Delta \neq 0 \iff \kappa^4 + \kappa_s^2 \neq 0.\]

At a simple isotropic point $t_0$, with $q(t) = a~(t-t_0) + O((t-t_0)^2)$, where $a\in\C\setminus\{0\}$  and $\Delta(t_0) \neq 0$. As $t \to t_0$ we have
\[
\kappa(t)^4 + \kappa_s(t)^2 = \frac{\Delta(t_0)^2}{a^6 (t-t_0)^6} \left( \Delta(t_0)^2 + \frac{9}{4} a^2 \right) + O\left((t-t_0)^{-5}\right).
\]
Factoring out $(t-t_0)^{-6}$ shows that $\kappa^4 + \kappa_s^2 = 0$ if and only if $\Delta(t_0)^2 + \frac{9}{4} a^2 = 0$. Therefore $\Delta(t_0) = \pm \frac{3}{2} i a$. Consequently, if the underlying curve satisfies $\Delta(t_0) = \pm \frac{3}{2} i a$, the versality condition $\kappa^4 + \kappa_s^2 \neq 0$ fails near the isotropic point, forcing a non-transverse degenerate section of the complex Swallowtail surface $\mathcal{D}_F$.
\end{proof}

\subsection*{Acknowledgments}  
We are grateful to Farid Tari and Amanda Falqueto for their careful and constructive review of our manuscript.
\subsection*{Declarations}
\paragraph{\textbf{Conflicts of interest:}} The authors declare no conflict of interest.

 \end{document}